\documentclass{article}

\usepackage{graphicx}
\usepackage{amsmath}
\usepackage{amssymb}
\usepackage{color}
\usepackage{hyperref}
\usepackage{epsfig}
\allowdisplaybreaks

\newtheorem{lemma}{Lemma}
\newtheorem{theorem}{Theorem}

\newtheorem{algo}{Algorithm}
\newtheorem{coro}{Corollary}
\newtheorem{defi}{Definition}
\newtheorem{remark}{Remark}
\newtheorem{cond}{Condition}
\newtheorem{prop}{Proposition}
\newtheorem{example}{Example}

\newcommand{\NN}{{\mathbb N}}

\newcommand{\QQ}{{\mathbb Q}}
\newcommand{\ZZ}{{\mathbb Z}}
\newcommand{\RR}{{\mathbb R}}
\newcommand{\FF}{{\mathbb F}}
\newcommand{\bsx}{\boldsymbol{x}}
\newcommand{\bsc}{\boldsymbol{c}}

\newcommand{\bsT}{{\bf T}}

\newenvironment{proof}{\begin{trivlist}
   \item[\hskip\labelsep{\it Proof.}]}{$\hfill\Box$\end{trivlist}}
\providecommand{\keywords}[1]{\textbf{{Keywords}} #1}
\providecommand{\MSC}[1]{
   \textbf{{Math.\ Subj.\ Class.\ (2010)}} #1}
\newcommand{\notiz}[1]{}

\newcommand{\vol}{{\rm vol}}

\title{\scshape{An extension of the digital method \\based on $b$-adic integers}}

\author{Roswitha Hofer\thanks{supported by the Austrian Science Fund (FWF):
Project F5505-N26, which is a part of the Special Research Program
``Quasi-Monte Carlo Methods: Theory and Applications''} \ and \'Isabel Pirsic
\thanks{supported by the Austrian Science Fund (FWF): Project F5511-N26,
which is a part of the Special Research Program ``Quasi-Monte Carlo Methods:
Theory and Applications'' as well as Project P27351-N26}}

\begin{document}

\maketitle

\begin{abstract}
We introduce a hybridization of digital sequences with uniformly distributed sequences
in the domain of $b$-adic integers, $\ZZ_{b}, b\in\NN\setminus\{1\}$, by using such
sequences as input for generating matrices. The generating matrices are then
naturally required to have finite row-lengths. We exhibit some relations of the
`classical' digital method to our extended version, and also give several examples of
new constructions with their respective quality assessments in terms of $t,\mathbf T$ and 
discrepancy.
\end{abstract}


\MSC{ 11J71, 11K16, 11K38, 11F85}

\section{Introduction}
Constructing sequences with good equidistribution properties is an important problem in number theory and has applications to quasi-Monte Carlo 
methods in numerical analysis (see, e.g., \cite{DP10, niesiam}). 
In this context, the star discrepancy appears as an important 
measure of uniform distribution. For a given dimension $s\geq 1$, let $J$ be a subinterval of $[0,1]^s$ and let $\bsx_0,\ldots, \bsx_{N-1}$ 
be $N$ points in $[0,1]^s$ (we speak also of a \emph{point set} $\mathcal{P}$ of $N$ points in $[0,1]^s$). We define the counting function $A$ for the interval $J$ by 
$A(J;\mathcal{P}):=\#\{0\leq n<N: \bsx_n\in J\}$. Then the \emph{star discrepancy} of the point set $\mathcal{P}$ consisting of the points
$\bsx_0,\ldots,\bsx_{N-1}$ is defined by 
$$D_N^*(\mathcal{P})=D_N^*(\bsx_0,\ldots, \bsx_{N-1})=\sup_{J}\left| \frac{A(J;\mathcal{P})}{N}-\vol(J)\right|,$$
where the supremum is extended over all subintervals $J$ of $[0,1]^s$ with one vertex at the origin. 

For a sequence $\mathcal{S}$ of points $\bsx_0,\bsx_{1},\ldots$ in $[0,1]^s$,
the star discrepancy of the first $N$ terms of $\mathcal S$ is defined as
$D_{N}^*(\mathcal{S})=D_N^*(\bsx_0,\ldots, \bsx_{N-1})$.  The sequence
$\mathcal{S}$ is called \emph{uniformly distributed} if and only if
$D^*_N(\mathcal{S})\to 0$ as $N\to \infty$.

We say that $\mathcal{S}$ is a \emph{low-discrepancy sequence} if
\begin{equation}\label{equ:DB}ND_N^*(\mathcal{S})=c(\log N)^s +O\left((\log
N)^{s-1}\right) \qquad \mbox{for all $N\geq 2$},\end{equation} where $c>0$ and
the implied constant do not depend on $N$.  It is conjectured that
$O(N^{-1}(\log N)^s)$ is the least possible order of magnitude in $N$ that can
be obtained for the star discrepancy of a sequence of points in $[0,1]^s$. 

The probably most widespread technique for constructing low-discrepancy
sequences is the  \textit{digital method} which was introduced by Niederreiter
\cite{N87} and later slightly generalized in \cite{NiXi96}.
\begin{algo}\label{algo:1}
Choose a dimension $s\in\NN$, a finite commutative ring $R$ with identity and
of order $b$, and set $Z_b=\{0,1,\ldots,b-1\}$. Choose 
\begin{enumerate}
\item[(i)] bijections $\psi_r:Z_b\to R$ for all integers $r\geq 0$, satisfying
 $\psi_r(0)=0$ for all sufficiently large $r$;
\item[(ii)] elements $c^{(i)}_{j,r} \in R$ for $1\leq i\leq s$, $j\geq 1$,
 $r\geq 0$; 
\item[(iii)] bijections $\lambda_{i,j}:R\to Z_b$ for $1\leq i\leq s$, 
 $j\geq 1$.
\end{enumerate} The $i$th component $\bsx_n^{(i)}$ of the $n$th point $\bsx_n$
of the sequence $(\bsx_n)_{n\geq 0}$ is defined using the base $b$ representation of
$n=\sum_{r=0}^\infty a_rb^r$ with $b_r\in Z_b$ and $a_r=0$ for all sufficiently
large $j$ as follows. 
\begin{equation}\label{eq:def1}
\bsx_n^{(i)}:=\sum_{j=1}^\infty \lambda_{i,j}\left(\sum_{r=0}^\infty c_{j,r}\psi_r(a_r)\right)/b^j.
\end{equation}
\end{algo}
Note that the inner sum in \eqref{equ:def1} is a finite sum because of the
choice that $\psi_r(0)=0$ for all sufficiently large $r$ and the fact that
$a_r=0$ for all sufficiently large $r$.

Usually the presentation of the digital method uses the concept of infinite \emph{generating matrices}, $C^{(i)}:=(c^{(i)}_{j,r})_{j\geq 1,r\geq 0}\in R^{\NN\times \NN_0}$ for $i\in\{1,\ldots,s\}$ with the construction given as follows. Set 
$$C^{(i)}\cdot\left(\begin{matrix}\psi_0(a_0)\\\psi_1(a_1)\\\vdots \end{matrix}\right)=:\left(\begin{matrix}y_{n,1}^{(i)}\\y_{n,2}^{(i)}\\\vdots \end{matrix}\right).$$
Then, 
$$\bsx^{(i)}_n=\sum_{j=1}^{\infty}\lambda_{i,j}(y^{(i)}_{n,j})b^{-j}.$$

Obviously the challenge is to find appropriate elements $c^{(i)}_{j,r}\in R$
such that the generated sequence $(\bsx_n)_{n\geq 0}$ is a low-discrepancy
sequence.

Most of the actual constructions choose the ring $R$ to be a finite field
$\FF_q$ with prime-power cardinality $q$. This has the advantage that basic
linear algebra is available and the distribution of the generated sequence
amongst elementary intervals is related to the rank-structure of the generating
matrices. 

An elementary interval in base $b$ is an interval $I\subset [0,1]^s$ of the form 
$$I=\prod_{i=1}^s\left[\frac{a_i}{b^{d_i}},\frac{a_i+1}{b^{d_i}}\right)$$
with nonnegative integers $d_i, 0\leq a_i<b^{d_i}$ for $i=1,\ldots ,s$.

The distribution amongst elementary intervals is relevant when determining the
quality-parameter function $\bsT$ or the quality parameter $t$ of the generated
sequence when it is considered as a $(\bsT,s)$-sequence in base $b$ in the
sense of Larcher and Niederreiter \cite{LN95} or a $(t,s)$-sequence in base $b$
in the sense of Niederreiter \cite{N87}. Those concepts later have been
modified by including truncation in \cite{Tez93,XiNi95,NieOez} in order to meet
certain requirements in special constructions. Throughout the paper
$[x]_{b,m}:=\sum_{j=1}^m x_jb^{-j}$ denotes the $m$-digit truncation of the
real $x\in[0,1]$ in base $b$ with a specific given base $b$ representation $
x=\sum_{j=1}^\infty x_jb^{-j}$, where the case that all but finitely many
$x_j=b-1$ is explicitly admissible as well. For a vector the base $b$ $m$-digit
truncation is applied by coordinates. 

\begin{defi}
Let $b,t,m$ be integers satisfying $b\geq 2$ and $0\leq t\leq m$. A \emph{$(t,m,s)$-net in base $b$} is a point set of $b^m$ points in $[0,1)^s$ such that every 
elementary interval $I \subseteq [0,1)^s$ in base $b$ with volume $b^{t-m}$ contains exactly $b^t$ points of the point set. 

Let $\bsT:\NN_0\to\NN_0$ satisfying $\bsT(m)\leq m$ for all $m\in\NN_0$. A
sequence $\bsx_0,\bsx_1,\ldots$ of points in $[0,1]^s$ is called a
\emph{$(\bsT,s)$-sequence in base $b$} if for all integers $k\geq 0$ and $m$
satisfying $\bsT(m)<m$ the points $[\bsx_n]_{b,m}$ with $kb^m\leq n<(k+1)b^m$
form a $(\bsT(m),m,s)$-net in base $b$.  As a special case, such a sequence is
called a \emph{$(t,s)$-sequence in base $b$} with $t\in\NN_0$ if it is a
$(\bsT,s)$-sequence in base $b$ with $\bsT(m)\leq t$ for all $m\geq 0$.

\end{defi}

A $(\bsT,s)$-sequence in base $b$ is uniformly distributed if
$\lim_{N\to\infty}(m-\bsT(m))=\infty$. In particular, every $(t,s)$-sequence is
uniformly distributed. Furthermore, if
$(\frac{1}{r}\sum_{m=1}^rb^{\bsT(m)})_{r\in\NN}$ is bounded, then the
$(\bsT,s)$-sequence in base $b$ is a low-discrepancy sequence. Consequently,
every $(t,s)$-sequence is a low-discrepancy sequence. 

It is well-known that the digital method in Algorithm~\ref{algo:1} applied to a
finite field with cardinality $q$ constructs a digital $(\bsT,s)$-sequence over
$\FF_q$ if and only if the following condition holds. 
\begin{cond}\label{cond:1}
For every 
integer $m$ satisfying $m>\bsT(m)$ and all nonnegative
 integers $d_1,\ldots,d_s \geq 0$ with $1 \leq d_1+\cdots + d_s \leq m-\bsT(m)$, the 
$(d_1+\cdots + d_s)\times m$ matrix over $\FF_q$ formed by the row vectors 
$$(c^{(i)}_{j,0},c^{(i)}_{j,1},\ldots,c^{(i)}_{j,m-1}) \in \FF_q^m$$
with $1 \le j \le d_i$ and $1 \le i \le s$, has rank $d_1+\cdots + d_s$. 
\end{cond}
Analogously, Algorithm \ref{algo:1} produces a $(t,s)$-sequence if Condition \ref{cond:1} holds with $\bsT(m)=t$ for $m\geq t$ and $\bsT(m)=m$ else. 
Note that Algorithm~\ref{algo:1} produces a uniformly distributed sequence if and only if the following condition holds. 
\begin{cond}\label{cond:ud}
For every choice of $d_1,\ldots,d_s\geq 0$ (not all zero) the rows $\bsc_j^{(i)}=(c^{(i)}_{j,r})_{r\geq 0}$, $1\leq j\leq d_i,\, 1\leq i\leq s$ are linearly independent. 
\end{cond}
For more details on $(\bsT,s)$-sequences and their digital versions we refer the interested reader to \cite{niesiam,DP10}.

For reasons related to the uniform distribution of mixed-base digital sequences
so-called finite-row generating matrices, i.e., matrices having in each row
only finitely many nonzero entries, have been the subject of investigation. We refer to
\cite{hklp,hoflar,HofPir11,hoFFA,HofPir12,hoMM,HofNie} for examples and
constructions of finite-row generating matrices and more about their motivation. Note that
the finite-row property of the matrices, i.e., for every $i=1,\ldots,s$ and
$j\geq 1$, $c^{(i)}_{j,r}=0$ for all sufficiently large $r$, ensures the
finiteness of the inner sum $\sum_{r=0}^\infty c_{j,r}\psi_r(a_r)$ in
\eqref{eq:def1}. Hence, when using finite-row generating matrices in
the digital method, any sequence of bijections $(\psi_r)_{r\geq 0} $ can be
used and the index sequence for the construction can in accordance be chosen freely as any
sequence of $b$-adic integers, i.e., $\ZZ_{b}$, instead of just the nonnegative integers.
Importantly, note that $b$ is \emph{not} required to be prime.

This yields the following alternative algorithm. 

\begin{algo}\label{algo:2}
Choose a dimension $s\in\NN$, a finite commutative ring $R$ with identity and of order $b$, and set $Z_b=\{0,1,\ldots,b-1\}$. Choose 
\begin{enumerate}

\item[(i)] bijections $\psi_r:Z_b\to R$ for all integers $r\geq 0$;
\item[(ii)] elements $c^{(i)}_{j,r} \in R$ for $1\leq i\leq s$, $j\geq 1$, $r\geq 0$, satisfying $c^{(i)}_{j,r}=0$ for all sufficiently large $r$ for fixed $i,j$; 
\item[(iii)] bijections $\lambda_{i,j}:R\to Z_b$ for $1\leq i\leq s$, $j\geq 1$.
\item[(iv)] a sequence $(s_n)_{n\geq 0}$ in $\ZZ_b$.  
\end{enumerate}
The $i$th component $\bsx_n^{(i)}$ of the $n$th point $\bsx_n$ of the sequence $(\bsx_n)$ is defined using the $b$-adic representation of $s_n=\sum_{r=0}^\infty a_rb^r$ with $a_r\in Z_b$ as follows. 
\begin{equation}\label{equ:def1}
\bsx_n^{(i)}=\sum_{j=1}^\infty \lambda_{i,j}\left(\sum_{r=0}^\infty c_{j,r}\psi_r(a_r)\right)/b^j.
\end{equation}
\end{algo}

The paper is organized as follows:
in Section 2 we recall definitions pertaining to $b$-adic numbers
 and introduce some lemmas. \\
Section 3 explains relations between our new, extended Algorithm 2 and
Algorithm 1 of the classical digital method.\\
Specific examples of new constructions and their quality assessments will
be given in Section 4.

\section{Relevant background on $b$-adic numbers}

An introduction in and construction of $b$-adic integers and numbers for arbitrary
integers $b\geq 2$ can, e.g.,  be found in
 \cite{Mahler}. Uniform distribution in the $b$-adic integers was introduced by Meijer \cite{Meijer},
 whose definitions we will employ and first briefly recall here. We use the name $b$-adic
 rather than $g$-adic, i.e., the letter $b$, to signify a not necessarily prime digit base, as is
 customary in the literature on uniform distribution modulo 1.

\subsection{$b$-adic numbers and integers}
(Detailed proofs for the claims in this section can be found in \cite{Meijer}.)

Analogously to the case of $p$-adic numbers, $b$-adic numbers can be introduced
as completion of $\QQ$, only in this case not by a valuation (using the definition of Meijer,
in the sense of an `absolute value'), but the following pseudo-valuation. 
\begin{defi}\label{bvaldef}
 Let $b\geq 2$ be a positive integer and $a\in\QQ$ a rational. 
 The prime decompositions of $b,a,$ shall be given as
 $$ b= p_{1}^{\beta_{1}}\dots p_{r}^{\beta_{r}}, \,
    a= \pm p_{1}^{\alpha_{1}}\dots p_{s}^{\alpha_{s}},\quad \alpha_{i}\in\ZZ;\,r,s,\beta_{i}\in\NN.$$
 Then the $b$-adic pseudo-valuation is defined by
 $$ |a|_{b} := \max_{i, p_{i}|b}\  b^{-\alpha_{i}/\beta_{i}},\,|0|_{b}:=0. $$
 
\end{defi}
The `pseudo' part signifies that the multiplicative identity demanded for a valuation
only holds up to inequality, i.e.,$$ |m n|_{b} \leq |m|_{b}|n|_{b}.$$
Nevertheless, $d(x,y):=|x-y|_{b}$ is a 
(non-archimedean) metric, so the following definition is valid.
\begin{defi}
 Let $b\geq2$ be a positive integer.
 The ring obtained by completion of $\QQ$ with respect to the $b$-adic pseudo-valuation shall be
 called  the ring $\QQ_{b}$ of $b$-adic numbers, and accordingly we define the subset
 \[ \ZZ_{b} :=\{ a : a\in\QQ_{b}, |a|_{b}\leq 1\}
 \]
 of $b$-adic integers.
\end{defi}
We remark the following observations:
\begin{enumerate}
\item Clearly, $\QQ\subset\QQ_{b},\, \ZZ\subset\ZZ_{b},$ and $\ZZ_{b}$ 
  is indeed a subring of $\QQ_{b}$.
\item $1/b<|a|_{b}\leq 1$ for $a\in\NN, 1\leq a<b$ where for composite $b$, values less than $1$ 
can indeed occur, e.g.,  $|6|_{24}=|18|_{24}=1/\sqrt[3]{24},\  |12|_{24}=1/\sqrt[3]{24^{2}}$.
 \item For $b$ prime, the definitions coincide with the usual notions. For
  composite $b$, we have the decomposition
  \[\QQ_{b} \cong \QQ_{p_{1}}\times\dots\times\QQ_{p_{r}},\]
  using the notation of Definition \ref{bvaldef}. 
 \item As in the $p$-adic case, each $a\in\QQ_{b}$ has a unique representation
 \[  a = \sum_{i=k_{0}}^{\infty} a_{i} b^{i},\quad k_0\in\ZZ,\, a_{i}\in\{0,\dots,b-1\},\, a_{k_{0}}\neq0, \]
 and $|a|_{b} = b^{-k_{0}} |a_{k_{0}}|_{b}$. For $b$-adic integers, $k_{0}$ is $0$, i.e., we get
 a representation as a formal power series in $b$. Furthermore, for all $a\in\NN_{0}$ the
 digit expansion in base $b$ and the $b$-adic representation coincide.

\item The number (in $\QQ$) obtained by truncation of the unique representation of a number
$a\in\QQ_{b}$ at index $k\in\ZZ$ is defined by $\tau_{k}(a):=\sum_{i=k_0}^{k-1}a_ib^i$. 
A $b$-adic integer $a$ is a unit if and only if $\gcd(\tau_1(a),b)=1$. Moreover, if $a$ is a unit then $|a|_b=|a^{-1}|_b=1$ (see \cite[Lemma~4]{Meijer}). 
\end{enumerate}

\subsection{Uniform distribution in $\ZZ_b$}

First we recall the definition of uniform distribution in the integers.
\begin{defi}
 Let $\omega=(x_{n})_{n\geq0}\in\NN_{0}^{\NN_{0}}$ be a sequence of nonnegative integers, and $m\in\NN, k\in\NN_{0},\, m>1$.
 If for any $a,\,0\leq a<m$ we have
 \[
  \lim_{N\to\infty}\frac{\#\{n : x_{n}\equiv a \bmod m,\, n< N\}}{N} = \frac1{m},
 \]
 $\omega$ is called uniformly distributed (u.d.) modulo $m$.
 
 If $\omega$ is u.d.\  modulo $m$ for {any} $m>1$ it is called
 uniformly distributed in $\ZZ$.
 \end{defi}
 
The $b$-adic case models this very closely.
However, here we require the `local uniformity'  only at powers of $b$.

\begin{defi}
 Let $\omega=(x_{n})_{n\geq0}\in\ZZ_{b}^{\NN_{0}}$ be a sequence of $b$-adic integers, and $k\in\NN_{0}$.
 If for any $a,\,0\leq a<b^{k}$ we have
 \[
  \lim_{N\to\infty}\frac{\#\{n : |x_{n}- a|_b \leq b^{-k},\, n< N\}}{N} = \frac1{b^{k}},
 \]
 $\omega$ is called $k$-uniformly distributed in $\ZZ_{b}$.
 
 If $\omega$ is $k$-uniformly distributed  for {any} $k\geq 1$ it is called
 uniformly distributed in $\ZZ_{b}$.
 \end{defi}

For reference we state the precise relation between the two notions as a lemma 
(Cf. Corollary 1 in \cite{MeijerShiue}). It is easily seen by
first observing that $\tau_{k}(x) \equiv a \bmod b^{k}\iff |x-a|_{b}\leq b^{-k}$ for
$x\in\ZZ_{b},a\in\NN_{0},\, a<b^{k},\,k\in\NN_{0}$.
\begin{lemma} \label{udrelation}
 A sequence $\omega=(x_{n})_{n\geq0}\in\ZZ_{b}^{\NN_{0}}$ is u.d.\  in $\ZZ_{b}$
 if and only if the sequences $(\tau_{k}(x_{n}))_{n\geq0}\in\NN_{0}^{\NN_{0}}$ are
 u.d.\  modulo $b^{k}$ for every $k\geq 1$.
\end{lemma}

As examples of uniformly distributed sequences in $\ZZ$ the following are listed in
\cite[Ch.5]{KN74}. By the previous definitions and lemma they can as well be regarded as u.d.\  in 
$\ZZ_{b}$ for any $b\geq 2$.
\begin{enumerate}
\item $(\lfloor\alpha n\rfloor)_{n\geq0}$ for irrational $\alpha\in\RR\setminus\QQ$.
\item $(\lfloor f(n) \rfloor)_{n\geq0}$ for $f\in\RR[x]$, where some coefficient apart from the
 constant is irrational.
 \item $(\lfloor\alpha n^{\sigma}\rfloor)_{n\geq0}$ for  $\alpha\in\RR,\sigma\in\RR^{+}\setminus\NN$. 
\end{enumerate}
From \cite[Th.2]{Meijer} we know that $(n)_{n\geq 0 }$ is u.d.\ in $\ZZ_{b}$.
An example of how to obtain new u.d.\ sequences from other u.d.\ sequences
 is also  given by  \cite[Th.3]{Meijer}. 
\begin{lemma}
 Let $a,c\in\ZZ_{b}$ and $(x_{n})_{n\geq0}$ be a sequence u.d.\ in $\ZZ_{b}$.
 Then the sequence $(ax_{n}+c)_{n\geq0}$ is also u.d.\ in $\ZZ_{b}$ if and only if $a$ is a unit.
\end{lemma}
Consequently, $(an+c)_{n\geq0}$ is u.d.\ in $\ZZ_{b}$, if $a$ is a unit.

\begin{example}\label{examp:square}
Let $a,c,d$ be $b$-adic integers such that $|a|_b<1$ and $c$ is a unit. Then the sequence $(an^2+cn+d)_{n\geq 0}$ is uniformly distributed in $\ZZ_b$. (See \cite[Theorem~5]{Meijer}). Furthermore the sequence $(n^2)_{n\geq 0}$ is not uniformly distributed in $\ZZ_b$, which can be seen by quadratic residues modulo $b^k$. 
\end{example}

\subsection{The $b$-adic representation of $b$-adic integers}

Obviously the $b$-adic representation of a nonnegative integer corresponds to its base $b$
digit expansion. This is not true for the negative integers, as is shown, e.g., by the $b$-adic 
representation of $-1=\sum_{i=0}^\infty (b-1)b^i$. We recall the general situation and add
some further details:

\begin{lemma}\label{lem:1}Let $n$ be a positive (rational or $b$-adic) integer. 
The $b$-adic representation of $-n$ is related to the $b$-adic representation (or expansion) of
 $n=\sum_{i=r}^\infty a_ib^i; r\in\NN_0,\,a_r\neq0$ via 
$$-n=(b-a_r)b^r+\sum_{i=r+1}^\infty (b-1-a_i)b^{i} .$$
(where $r=\min\{r'\in\NN_0:a_{r'}\neq 0\}$, i.e., $|n|_{b}=|a_{r}|_{b}b^{-r}.$) 

Let $k,l\in\NN_0$ and $M=\{-n:kb^l<n\leq (k+1)b^l\}$. 
The $b$-adic integers in $M$ share the same digits in their $b$-adic representations
from index $l$ on and run through all possible values in their first $l$ digits

We denote the $b$-adic representation of a $b$-adic integer $z$ by $z=\sum_{i=0}^\infty \overline{a}_ib^i$.
Then $(\overline{a}_l,\overline{a}_{l+1},\overline{a}_{l+2},\ldots)\in Z_b^\infty$ is equal for every integer in $M$ and $\#\{(\overline{a}_0,\ldots,\overline{a}_{l-1}):z\in M\}=\#M=b^l$. 
(i.e., $M$ is the equivalent of a $b$-adic block in negative $b$-adic integers).
\end{lemma}
\begin{proof}
The first part is easily verified by simply adding $\tau_k(n)$ and $\tau_k(-n)$ 
(i.e., insert the given representation for $-n$) and observing
that the sum, $b^{k}$, converges to $0$ in the $b$-adic metric as $k$ goes to infinity, while the
summands converge to $n$ and $-n$, respectively.

Regarding the second part first note that for $n$ in the given range
we always have $r=r(n)\leq l$ and $r=l$ iff $n= (k+1)b^l$. We denote the $b$-adic 
expansion (or representation) of $k$ by $k=\sum_{i\geq 0} k_ib^i$.

For $r<l$ it is evident that the digits of all $n$ in the given range 
are constant from index $l$ onwards and in fact equal to $k_l,k_{l+1},\dots$~\ . 
Thus the digits of $-n$ with index at least $l$ are also
constant and equal to $b-1-k_l,b-1-k_{l+1},\dots$~\ . Furthermore the mapping between the
(full range of) digits of $n$ and $-n$ is self-inverse. In particular it is injective on the
first $l$ digits. There is however no $n$ with $r(n)<l$ such that all of the first $l$ digits are
$0$. Therefore the statement is proved, if we can show that the remaining case, $r=l$,
$n=(k+1)b^l$, maps to the zero vector in the first $l$ digits and, more importantly, to the
same trailing digits.

For this case we set $r'\in\NN_0$ such that $k_{l}=k_{l+1}=\dots=k_{l+r'-1}=b-1\neq k_{l+r'}$
(where the case $r'=0$ may occur if $k_l\neq b-1$). Then, obeying the carry,
\[(k+1)b^l = (k_l+1)b^l+\sum_{i>0}k_{l+i}b^{l+i}= (k_{l+r'}+1)+\sum_{i>0}k_{l+r'+i}b^{l+r'+i},\]
so the last right hand side is again a valid expansion (representation). 
By our proposed formula this maps to
\[ (b-(k_{l+r'}+1))b^{l+r'}+\sum_{i>0}(b-1-k_{l+r'+i})b^{l+r'+i} = \sum_{i\geq0}(b-1-k_{l+i})b^{l+i},\]
giving the desired form and concluding the argument.

\end{proof}

\section{Relations between Algorithms \ref{algo:1} and \ref{algo:2}}

\begin{theorem}\label{thm:ud1} Let $s$ be a dimension, $R=\FF_q$ be a finite field with cardinality $q$, and $C^{(1)},\ldots,C^{(s)}\in \FF_q^{\NN\times \NN_0}$ be finite-row generating matrices. Let $(s_n)_{n\geq 0}$ be a uniformly distributed sequence in $\ZZ_q$. 

If $C^{(1)},\ldots,C^{(s)}$ generate a uniformly distributed sequence via Algorithm \ref{algo:1}, then Algorithm \ref{algo:2} based on these matrices and the sequence $(s_n)_{n\geq 0}$ gives a uniformly distributed sequence in $[0,1]^s$. 
\end{theorem}
\begin{proof}
To prove the uniform distribution we show that every elementary interval contains the correct portion of points in the limit. Let 
$$I=\prod_{i=1}^s\left[\frac{a_i}{q^{d_i}},\frac{a_i+1}{q^{d_i}}\right)$$
with nonnegative integers $d_1,\ldots,d_s$ (not all zero) and $a_1<q^{d_1},\ldots,a_s<q^{d_s}$. The finite-row property ensures that there exists an $L$ such that $c^{(i)}_{j,r}=0$ for all $i=1,\ldots,s$, $1\leq j\leq d_i$ and $r\geq L$. Hence $\tau_L(s_n)$ determines whether $\bsx_n$ of Algorithm \ref{algo:2} is included in $I$ or not.  Since Condition~\ref{cond:ud} holds true we know that exactly $q^{L-(d_1+\cdots +d_s)}$ residue classes $r_1,\ldots,r_{q^{L-(d_1+\cdots +d_s)}}$ modulo $q^L$ correspond to the elementary interval $I$. The uniform distribution of $(s_n)_{n\geq 0}$ in $\ZZ_q$ ensures that $(\tau_L(s_n))_{n\geq 0}$ is uniformly distributed modulo $q^L$. Therefore 
\begin{align*}
\frac{A(N;I)}{N}-\lambda(I)&=\\
\frac1N\!\! \sum_{k=1}^{q^{L-(d_1+\cdots +d_s)}}\hspace{-.7cm}
\#\{0\leq n<N:\tau_L(s_n)&\equiv r_k\pmod{q^L}\}
-\frac{1}{q^{d_1+\cdots +d_s}}\\
\intertext{and thus}
\lim_{N\to\infty}\left(\frac{A(N;I)}{N}-\lambda(I)\right)&=q^{L-(d_1+\cdots +d_s)}\frac{1}{q^L}-\frac{1}{q^{d_1+\cdots +d_s}}=0.
\end{align*}
\end{proof}

The converse of Theorem~\ref{thm:ud1} is not true. Let $s_n=n^2$ then the sequence $(s_n)_n\geq 0$ is not uniformly distributed in $\ZZ_q$ by Example~\ref{examp:square}. But the one-dimensional sequence generated by the finite-row matrix
$$C^{(1)}= \begin{pmatrix}
1&1&0&0&0&0&0&0&\dots \\
0&0&1&1&0&0&0&0&\dots \\
0&0&0&0&1&1&0&0&\dots \\
\vdots&\vdots&\vdots&\vdots&\vdots&\vdots&\vdots&\vdots&\vdots&\ddots 
\end{pmatrix}\in\FF_2^{\NN\times\NN_0}$$ 
via Algorithm~\ref{algo:2} is a uniformly distributed digital sequence in base $2$ (see \cite{HoLaZe}). In the next theorem we will find conditions on the generating matrices such that the uniform distribution of $(s_n)_{n\geq 0}$ is sufficient and necessary for the uniform distribution of the digital sequence. 
For stating the next result we introduce the term \emph{optimal row-lengths} for the generating 
matrices. Let $C_1,\ldots,C_s\in\FF_q^{\NN\times \NN_0}$ be finite-row generating matrices. We 
measure the length of a row $\bsc^{(i)}_j$ by $\sup\{r\in\NN_0:c^{(i)}_{j,r}\neq 0\}+1$. In 
\cite{hoflar} it was shown that the uniform distribution of the sequence implies that for every $j$ 
there exists $i\in\{1,\ldots,s\}$ such that the length of $\bsc_j^{(i)}$ is at least $s j$. 
Therefore we say 
that $C_1,\ldots,C_s\in\FF_q^{\NN\times \NN_0}$ have \emph{optimal row-lengths} if for all 
$i\in\{1,\ldots,s\},\,j\geq 1$ the length of $\bsc_j^{(i)}$ is \emph{at most} $s j$.

\begin{theorem}\label{thm:ud2}
Let $s>0$, $R=\FF_q$ be a finite field with cardinality $q$, and $C^{(1)},\ldots,$ $C^{(s)}\in \FF_q^{\NN\times \NN_0}$ be finite-row generating matrices having optimal row-lengths and yielding a $(\bsT,s)$-sequence with optimal quality parameter $\bsT\equiv 0$ via Algorithm~\ref{algo:1}. Let $(s_n)_{n\geq 0}$ be a sequence in $\ZZ_q$. 
Then Algorithm~\ref{algo:2} produces a uniformly distributed sequence if and only if $(s_n)_{n\geq 0}$ is uniformly distributed in $\ZZ_q$. 
\end{theorem}
\begin{proof}
Sufficiency of the uniform distribution of $(s_n)_{n\geq 0}$ in $\ZZ_q$ follows from Theorem~\ref{thm:ud1}. To prove necessity we regard elementary intervals of the following form 
\begin{equation}\label{equ:prnec}
I=\prod_{i=1}^s\left[\frac{a_i}{q^d},\frac{a_i+1}{q^d}\right),
\end{equation}
where $d\in\NN$ and $0\leq a_i<q^{d}$ for $i=1,\ldots,s$. We denote the base $q$ representation of $\frac{a_i}{q^d}$ by $$\frac{a_i}{q^d}=\frac{a_{i,1}}{q}+\frac{a_{i,2}}{q^2}+\cdots +\frac{a_{i,d}}{q^d}$$ and the $q$-adic representation of $s_n$, for some fixed $n$,
by $$s_n=\sum_{r=0}^\infty \overline{a}_rq^r.$$
The optimal row-lengths yield the equivalence of $\bsx_n\in I$ and 
$$ 
\underbrace{\begin{pmatrix}
c^{(1)}_{1,0} & \cdots & c^{(1)}_{1,ds-1} \\
\vdots & & \vdots  \\
c^{(1)}_{d,0} & \cdots & c^{(1)}_{d_1,ds-1} \\
\vdots & & \vdots  \\
c^{(s)}_{1,0}  & \cdots & c^{(s)}_{1,ds-1} \\
\vdots & & \vdots  \\
c^{(s)}_{d,0} & \cdots  & c^{(s)}_{d_s,ds-1}
\end{pmatrix}}_{:=\overline{C}}
\begin{pmatrix}\psi_0(\overline{a}_0)\\\vdots\\\psi_{ds-1}(\overline{a}_{ds-1}) 
\end{pmatrix}
= \begin{pmatrix}\lambda_{1,1}^{-1}(a_{1,1})\\\vdots \\ \lambda_{1,d_1}^{-1}(a_{1,d})\\ \vdots \\ \lambda_{s,1}^{-1}(a_{s,1})\\\vdots \\ \lambda_{s,d_s}^{-1}(a_{s,d} ) 
\end{pmatrix}  \in\FF_q^{ds},
$$
with the maps $\psi_{i},\lambda_{i,j}$ of Algorithm \ref{algo:1}.
Note that $\overline{C}$ is a square matrix with full row-rank $ds$, and that there are $q^{ds}$ elementary intervals of the form \eqref{equ:prnec} which are in one to one correspondence with the elements in $\FF_q^{ds}$. Further, note that the value of $\tau_{sd}(s_n)$ determines whether $\bsx_n$ is included in $I$ or not. Now the uniform distribution of $(\bsx_n)_{n\geq 0}$ ensures that $(\tau_{sd}(s_n))_{n\geq 0}$ is uniformly distributed modulo $q^{sd}$. Since this is valid for every $d\in\NN$, we have that $(\tau_{k}(s_n))_{n\geq 0}$ is uniformly distributed modulo $q^{k}$ for every $k\in\NN$  (clearly, uniform distribution propagates to lower powers). Finally by Lemma 
\ref{udrelation} this is equivalent to $(s_n)_{n\geq 0}$ being uniformly distributed in $\ZZ_q$.
\end{proof}

\begin{example}
Constructions and examples of generating matrices having optimal row-lengths and satisfying Condition~\ref{cond:1} for $\bsT\equiv 0$ are given e.g. in \cite{hoflar,HofPir11,hoFFA}. 

One specific  example of suitable matrices are the Stirling matrices, constructed in
analogy to the classical Pascal or Faure matrices but with Stirling numbers
of the first kind replacing the binomials. Depicted here are matrices in
base 5. The finite row length is clearly visible. The apparent fractal
structure and other aspects are explored in more detail in \cite{HofPir11}.

\noindent
\includegraphics[width=\textwidth]{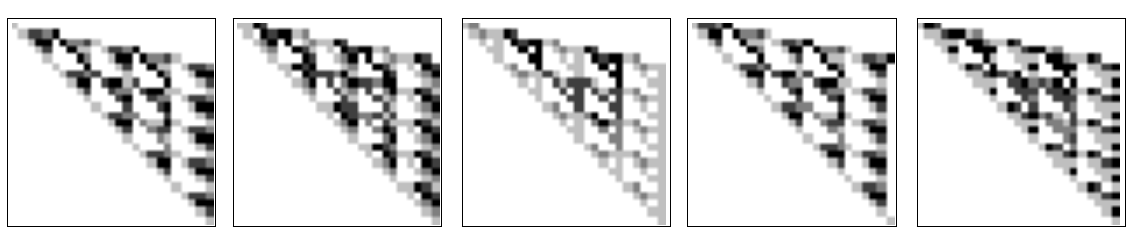}

\end{example}

\begin{remark}
Theorem~\ref{thm:ud2} in case of $s=1$ and $C^{(1)}$ chosen to be the identity matrix represents a generalization of the one-dimensional case of \cite[Theorem~4.2.(iv)]{HelNie}. 
\end{remark}

\section{Discrepancy and quality-parameter function $\bsT$ for Algorithm~\ref{algo:2}
using some specific $(s_n)_{n\geq 0}$}

The probably most basic setting for $(s_n)_{n\geq 0}$ is to choose $s_n=n$. Trivially,
{in this case} Condition \ref{cond:1} is also qualified to determine the quality parameter or function,
 $t$ or $\bsT$, respectively, of the sequence constructed by Algorithm \ref{algo:2}. We next look at 
 negative integers as the underlying sequence.

\begin{prop}[The case $s_n=-n-1$ :]\label{prop:1} Let $s$ be a dimension, $q$ be a prime power, $R=\FF_q$, and $C^{(1)},\ldots,C^{(s)}\in\FF_q^{\NN\times \NN_0}$ be finite-row generating matrices. 

If $C^{(1)},\ldots,C^{(s)}$ generate a $(\bsT,s)$-sequence in base $q$ via Algorithm \ref{algo:1}, then Algorithm \ref{algo:2} based on these matrices and the sequence $(s_n)_{n\geq 0}=(-n-1)_{n\geq 0}$ gives a $(\bsT,s)$-sequence in base $q$. 
\end{prop}
\begin{proof}Let $k,m$ be any nonnegative integers such that $\bsT(m)<m$. We have to prove that $[\bsx_n]_{q,m}$ with $n=kb^m,kb^m+1,\ldots,kb^m+b^m-1$ forms a $(\bsT(m),m,s)$-net in base $q$. 
We regard an elementary interval of the form 
$$I=\prod_{i=1}^s\left[\frac{a_i}{q^{d_i}},\frac{a_i+1}{q^{d_i}}\right)$$
with integers $d_i\geq 0,\,0\leq a_i<q^{d_i}$ for $i=1,\ldots,s$ satisfying $d_1+\cdots+d_s=m-\bsT(m)$. We use the base $q$ digit expansion  
${a_i}/{q^{d_i}}=\sum_{j=1}^\infty \frac{a_{i,j}}{q^j}$.

Regarding Algorithm \ref{algo:2} we see that $[\bsx_n]_{q,m}\in I$
(for some fixed $n$) if and only if 
\begin{align*}
\begin{pmatrix}
c^{(1)}_{1,0} & c^{(1)}_{1,1} & \cdots \\
\vdots & \vdots &  \\
c^{(1)}_{d_1,0} & c^{(1)}_{d_1,1} & \cdots \\
\vdots & \vdots &  \\
c^{(s)}_{1,0} & c^{(s)}_{1,1} & \cdots \\
\vdots & \vdots &  \\
c^{(s)}_{d_s,0} & c^{(s)}_{d_s,1} & \cdots 
\end{pmatrix}
\begin{pmatrix}\psi_0(\overline{a}_0)\\\psi_{1}(\overline{a}_1)\\\vdots 
\end{pmatrix}
=& \begin{pmatrix}\lambda_{1,1}^{-1}(a_{1,1})\\\vdots \\ \lambda_{1,d_1}^{-1}(a_{1,d_1})\\ \vdots \\ \lambda_{s,1}^{-1}(a_{s,1})\\\vdots \\ \lambda_{s,d_s}^{-1}(a_{s,d_s} )
\end{pmatrix},
\end{align*}
where we used the notation $-n-1=\sum_{r=0}^\infty \overline{a}_rq^r$. Equivalently, 
\begin{align*}&
\begin{pmatrix}
c^{(1)}_{1,0} & \cdots & c^{(1)}_{1,m-1} \\
\vdots & & \vdots  \\
c^{(1)}_{d_1,0} & \cdots & c^{(1)}_{d_1,m-1} \\
\vdots & & \vdots  \\
c^{(s)}_{1,0}  & \cdots & c^{(s)}_{1,m-1} \\
\vdots & & \vdots  \\
c^{(s)}_{d_s,0} & \cdots  & c^{(s)}_{d_s,m-1}
\end{pmatrix}
\begin{pmatrix}\psi_0(\overline{a}_0)\\\vdots\\\psi_{m-1}(\overline{a}_{m-1}) 
\end{pmatrix}=\\
=& \begin{pmatrix}\lambda_{1,1}^{-1}(a_{1,1})\\\vdots \\ \lambda_{1,d_1}^{-1}(a_{1,d_1})\\ \vdots \\ \lambda_{s,1}^{-1}(a_{s,1})\\\vdots \\ \lambda_{s,d_s}^{-1}(a_{s,d_s} ) 
\end{pmatrix}
- 
\begin{pmatrix}
c^{(1)}_{1,m} & c^{(1)}_{1,m+1} & \cdots \\
\vdots & \vdots &  \\
c^{(1)}_{d_1,m} & c^{(1)}_{d_1,m+1} & \cdots \\
\vdots & \vdots &  \\
c^{(s)}_{1,m} & c^{(s)}_{1,m+1} & \cdots \\
\vdots & \vdots &  \\
c^{(s)}_{d_s,m} & c^{(s)}_{d_s,m+1} & \cdots 
\end{pmatrix}
\begin{pmatrix}\psi_m(\overline{a}_m)\\\psi_{m+1}(\overline{a}_{m+1})\\\vdots 
\end{pmatrix}.
\end{align*}
Lemma \ref{lem:1} ensures that all terms on the right hand side are independent of $n$ for
 the range $n=kb^m,kb^m+1,\ldots,kb^m+b^m-1$ and also that the vector on the left hand side 
 spans $\FF_q^m$ as $n$ ranges through $n=kb^m,kb^m+1,\ldots,kb^m+b^m-1$ . Since Condition \ref{cond:1} holds we know that the system has exactly $q^{m-\bsT(m)}$ solutions for $n=kb^m,kb^m+1,\ldots,kb^m+b^m-1$. 
\end{proof}

Proposition~\ref{prop:1} immediately implies the following corollary. 

\begin{coro}[The case $s_n=(-1)^n\lfloor (n+1)/{2}\rfloor$ :] Let $s$ be a dimension, $q$ be a prime power, $R=\FF_q$, and $C^{(1)},\ldots,C^{(s)}\in\FF_q^{\NN\times \NN_0}$ be finite-row generating matrices. 

If $C^{(1)},\ldots,C^{(s)}$ generate a $(\bsT,s)$-sequence in base $q$ via Algorithm \ref{algo:1}, then Algorithm \ref{algo:2} with the same matrices and the underlying sequence 
$(s_n)_{n\geq 0}=((-1)^n\lfloor \frac{n+1}{2}\rfloor)_{n\geq 0}$ 
produces a sequence where the two subsequences $(\bsx_{2n})_{n\geq 0}$ and 
$(\bsx_{2n+1})_{n\geq 0}$ are $(\bsT,s)$-sequences in base $q$.

\end{coro}

For stating the next result we introduce the magnitude $\Delta_b(t,m,s)$, 
which is an upper bound for $D^{*}_{N}(\mathcal P)$
holding for every $(t,m,s)$-net $\mathcal{P}$ in base $b$:
$$b^m D_{b^m}^*(\mathcal{P})\leq \Delta_b(t,m,s).$$
In this paper we do not aim to give the most precise estimate but
include for the sake of completeness exemplarily the well-known
bound of Niederreiter for $b>2$ (see \cite[Th.4.5]{niesiam}):
\[
 \Delta_b(t,m,s) = b^t \sum_{i=0}^{s-1} {s-1 \choose i}{m-t\choose i}
 \left\lfloor \frac{b}2 \right\rfloor^i.
\]

\begin{prop}[The case $s_n=n+\alpha$ :]\label{prop:2}
Let $s$ be a dimension, $q$ be a prime power, $R=\FF_q$, and $C^{(1)},\ldots,C^{(s)}\in\FF_q^{\NN\times \NN_0}$ be finite-row generating matrices satisfying the quality parameter function $\bsT:\NN_0\to\NN_0$ in Condition \ref{cond:1}. Furthermore, let $\alpha$ be a $p$-adic integer with representation 
$\alpha=\sum_{i=0}^\infty a_iq^i$ and set $(s_n)_{n\geq 0}=(n+\alpha)_{n\geq 0}$. Then the discrepancy of the first $N$ points of the sequence $\mathcal{S}$ produced by Algorithm \ref{algo:2} satisfies
\begin{align*}
ND_N(\mathcal{S}) \leq & (q-a_{0}) \Delta_q(\bsT(0),0,s)+\sum_{j=1}^{r-1}(q-1-a_j)\Delta_q(\bsT(j),j,s)\\
& \quad \quad \, + \sum_{j=0}^{r}b_j\Delta_q(\bsT(j),j,s)   ,
\end{align*}
where $r=\lfloor \log_q(N) \rfloor$ and $\sum_{j=0}^r b_jq^j$ is the base $q$ representation of $N'=N-q^r+\sum_{j=0}^{r-1}a_jq^j$. 
\end{prop}
\begin{proof}
First collect the digits $a_j$ with $1\leq j\leq r-1 $ of $\alpha$ satisfying $a_j\neq q-1$ and denote them by 
$a_{k(1)}, a_{k(2)},\ldots,a_{k(j')}$ where $1\leq k(1)< k(2)<\ldots<k(j')\leq r-1$ and $j'$ is the number of such digits. Observe that $q^r-\sum_{i=0}^{r-1}a_iq^i=q-a_0+\sum_{i=1}^{j'}(q-1-a_{k(i)})q^{k(i)}$. 

Bearing in mind the basic fact that if dividing a point set $\mathcal{P}$ of $N$ points into $l$ disjoint sets $\mathcal{P}_1$ with $N_1$ points, ..., $\mathcal{P}_l$ with $N_l$ points, then 
$ND_N(\mathcal{P})\leq \sum_{w=1}^lN_wD_{N_w}(\mathcal{P}_w)$ --- we divide the first $q^r-\sum_{i=0}^{r-1}a_iq^i$ points as follows. 

We start with the first $q-a_0$ points which are obviously $q-a_0$
instances of $(T(0),0,s)$-nets in base $q$. 

For the next step we observe that the $q$-adic expansions of the next $q^{k(1)}$ points have all equal digits at position $k(1)$ and larger, and the digits at positions $0$ to $k(1)-1$ span $\{0,1,\ldots,q-1\}^{k(1)}$. Hence they form via Algorithm~\ref{algo:2} a $(T(k(1)),k(1),s)$-net in base $q$. 
We have $n$ in the range $$q-a_0,q-a_0+1,\ldots,q-a_0+q^{k(1)}-1$$
and thus $n+\alpha$ of the form
$$A_1,1+A_1,\ldots,q^{k(1)}-1+A_1$$
 with $A_1:=(a_{k(1)}+1)q^{k(1)}+\sum_{i=k(1)+1}^\infty a_iq^i$. 
Now analogous argumentation as in the proof of Proposition~\ref{prop:1} imply that these $q^{k(1)}$ points form a $(\bsT(k(1)),k(1),s)$-net in base $q$. Altogether we obtain step by step $(q-1-a_{k(1)})$ nets of this form. 
Absolutely identically we obtain $q-1-a_{k(2)}$, $(\bsT(k(2)),k(2),s)$-nets in base $q$, ..., and  $q-1-a_{k(j')}$, $(\bsT(k(j')),k(j'),s)$-nets in base $q$. 
This explains the first two terms in the upper bound. 

It remains to estimate the discrepancy of the residual $N'=N-q^r+\sum_{j=0}^{r-1}a_jq^j$ points. We denote the base $q$ representation of $N'$ by $N'=\sum_{j=0}^r b_jq^j$. Now we collect the digits 
$b_j$ with $1\leq j\leq r $ of $N'$ satisfying $b_j\neq 0$ and denote them by 
$b_{k(1)}, b_{k(2)},\ldots,b_{k(j'')}$, where $1\leq k(1)< k(2)<\ldots<k(j'')\leq r$, and $j''$ is the number of such digits. 

We regard the next $q^{k(j'')}$ points, i.e. $n$ in the range 
$$q^r-\sum_{i=0}^{r-1}a_iq^i,q^r-\sum_{i=0}^{r-1}a_iq^i+1,\ldots,q^r-\sum_{i=0}^{r-1}a_iq^i +q^{k(j'')-1}.$$ Now adding $\alpha$ yields 
$$A_2,1+A_2,\ldots, q^{k(j'')}-1+A_2$$
where $A_2:=(a_r+1)q^r+\sum_{i=r+1}^\infty a_iq^i$. 
As above the $b$-adic digits for powers with exponents at least $k(j'')$ are fixed and the first $k(j'')$ digits span $\{0,1,\ldots,q-1\}^{k(j'')}$. Hence those point sets form a $(T(k(j'')),k(j''),s)$-net in base $q$. 

Step by step we obtain $b_{k(j'')}$ such nets and we obtain the first summand in the second sum of the theorem. Again step by step, we obtain the last sum of the upper bound. 
\end{proof}

\begin{coro}\label{coro:2}
If, furthermore, $\bsT(m)\leq t$ with $t\in\NN_0$ in Proposition~\ref{prop:2} then 
$\mathcal{S}$ is a low-discrepancy sequence. 
\end{coro}
\begin{proof}
In this case we have $\Delta_q(t,j,s)\leq c_{q,s,t}j^{s-1}+\mathcal{O}(j^{s-2})$ for all $j\geq 0$, where the implied constant is independent of $j$. (See \cite{FaKr} for the best known constant $c_{q,s,t}$.) Using this bound in the upper bound of Proposition~\ref{prop:2} we obtain the desired result 
$ND_N(\mathcal{S})=O((\log N)^{s})$.
\end{proof}

\begin{coro}[$s_n=\frac{1}{v}n+\alpha$]\label{coro:3}
Let $s$ be a dimension, $q$ be a prime power, $t$ be  nonnegative integer, $R=\FF_q$, and $C^{(1)},\ldots,C^{(s)}\in\FF_q^{\NN\times \NN_0}$ be finite-row generating matrices satisfying the quality parameter function $\bsT\equiv t$ in Condition \ref{cond:1}. Furthermore, let $\alpha$ be a $p$-adic integer and $v\in\NN$ satisfying $\gcd(v,q)=1$. Set $(s_n)_{n\geq 0}=(\frac{1}{v}n+\alpha)_{n\geq 0}$. Then the sequence $\mathcal{S}$ produced by Algorithm \ref{algo:2} is a low-discrepancy sequence. 
\end{coro}
\begin{proof}
The strategy is to split the sequence into $v$ subsequences $$(s^{(1)}_n)_{n\geq 0},(s^{(2)}_n)_{n\geq 0},\ldots,(s^{(1)}_n)_{n\geq 0}$$ of the form 
\begin{align*}
s^{(1)}_n &= n+\alpha=:n+\alpha_1\\
s^{(2)}_n &= n+\frac{1}{v}+\alpha=:n+\alpha_2\\
& \vdots \\
s^{(v)}_n &= n+ \frac{(v-1)}{v}+\alpha=:n+\alpha_v.
\end{align*}
Since $\gcd(q,v)=1$ all of the fractions and thus all of $\alpha_1,\alpha_2,\ldots,\alpha_v$ are 
 $q$-adic integers.
 Then all $v$ subsequences are low-discrepancy sequences by Corollary~\ref{coro:2} and the result follows.

\end{proof}

\begin{remark}
Previously, subsequences of digital sequences produced by Algorithm~\ref{algo:1} have been investigated, see, e.g., \cite{hklp,hoferMCQMC,HoLaZe,HoZe}. In particular, subsequences indexed by arithmetic progressions were discussed in \cite{hklp,hoferMCQMC}. Unfortunately, the discrepancy of such sequences is a difficult subject and there exist negative results such as \cite[Example~5]{hoferMCQMC}. Hence the obvious generalization of Corollary~\ref{coro:3} to 
$s_n=\frac{u}{v}n+\alpha$ would be a difficult task as well. \end{remark}

\begin{example}

The following plots give a comparison between three different input sequences:
first, the classical sequence $s_n=n$, then the alternating sequence
$s_n=(-1)^n\lfloor (n+1)/{2}\rfloor$ and finally the sequence $s_n=(2n-1)/4$.

The generating matrices are the first two Stirling matrices in base 5 seen in
Example 1 and the number of points is 500.

\noindent
\includegraphics[width=\textwidth]{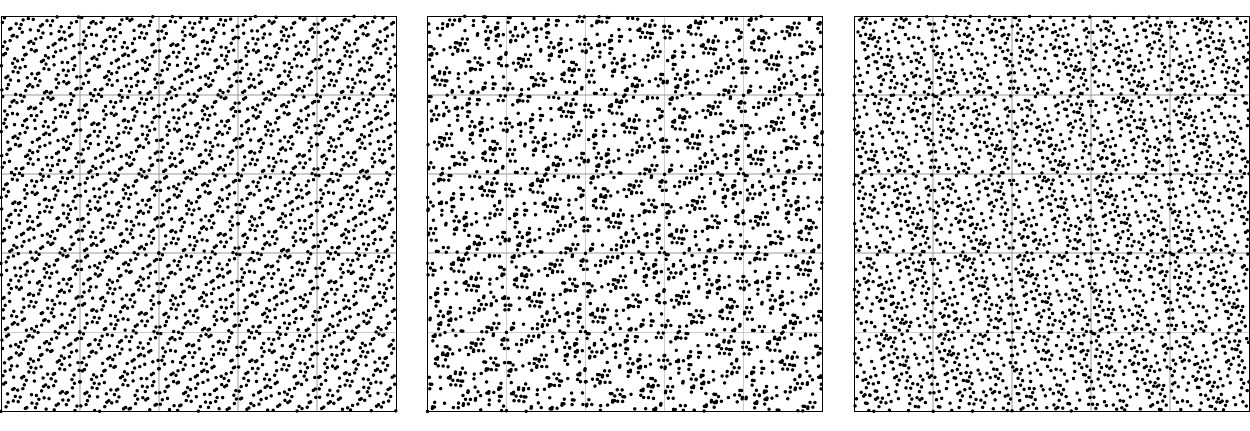}
\end{example}

\end{document}